\documentclass{amsart}

\usepackage[usenames]{color}

\usepackage{pb-diagram, soul}

\input xy
\xyoption{all}

\newcommand{\IR}{\mathbb R}

\newcommand{\F}{\mathcal F}

\newcommand{\DD}{\mathcal{D}}

\newcommand{\NN}{\mathbb{N}}
\newcommand{\FF}{\mathcal{F}}

\newcommand{\Nn}{\mathcal{N}}

\newcommand{\supp}{\mathrm{supp}}

\newtheorem{theorem}{Theorem}[section]
\newtheorem{lemma}[theorem]{Lemma}
\newtheorem{corollary}[theorem]{Corollary}

\newtheorem{proposition}[theorem]{Proposition}

\newtheorem{example}[theorem]{Example}
\newtheorem{fact}[theorem]{Fact}

\theoremstyle{definition}

\newtheorem{remark}[theorem]{Remark}

\begin{document}

\title[On topological properties of the weak topology of a Banach space]{On topological properties of the weak topology of a Banach space}

\author{S. Gabriyelyan, J. K{\c{a}}kol, L. Zdomskyy}

\thanks{The first named author was partially supported by ISF grant 441/11}
\address{Department of Mathematics, Ben-Gurion University of the Negev, Beer-Sheva, P.O. 653, Israel}
\email{saak@math.bgu.ac.il}

\thanks{The second named author was  supported by Generalitat Valenciana,
Conselleria d'Educaci\'{o}, Cultura i Esport, Spain, Grant PROMETEO/2013/058.
Part of this  paper  has been prepared when 
he visited the Kurt G\"odel
Research Center for Mathematical Logic at the  University of Vienna in October 2014.}
\address{Faculty of Mathematics and Informatics, A. Mickiewicz University, $61-614$ Pozna{\'n}, Poland}
\email{kakol@amu.edu.pl}

\thanks{The third-listed author was supported by the FWF grant I 1209-N25 as well as by the
Austrian Academy of Sciences  through the APART Program.}
\address{Kurt G\"odel Research Center for Mathematical Logic,
 University of Vienna,  W\"ahrin\-ger Str. 25, 1090 Vienna, Austria}
\email{lyubomyr.zdomskyy@univie.ac.at}
\urladdr{http://www.logic.univie.ac.at/\~{}lzdomsky/}

\subjclass[2000]{Primary 46A03,  54E18; Secondary 54C35, 54E20}

\keywords{weak topology, Banach space, $\aleph$-space, $k$-space, $cs^\ast$-character}

\begin{abstract}
Being motivated by the famous Kaplansky theorem we study various sequential properties
 of a Banach space $E$ and its closed unit ball $B$, both  endowed with the weak topology of $E$.
 We show that $B$ has the Pytkeev property if and only
if $E$ in the norm topology contains no isomorphic copy of $\ell_1$,
while  $E$ has the Pytkeev property if and only if it is finite-dimensional.
 We extend Schl\"{u}chtermann and Wheeler's result from \cite{S-W} by showing
that $B$ is a (separable) metrizable space if and only if it has countable
$cs^\ast$-character and is a $k$-space. As a corollary we obtain that $B$ is Polish
 if and only if it has countable $cs^\ast$-character and is \v{C}ech-complete,
that supplements  a result of Edgar and Wheeler \cite{edgar}.
\end{abstract}


\maketitle

\section{Introduction}

Topological properties of a  locally convex space (lcs) $E$ endowed with the weak topology
$\sigma(E,E')$, denoted by $E_w$ for short, are of  great importance and have been intensively
 studied from many years (see  \cite{kak,bonet} and references therein).
Corson \cite{Corson} started a systematic study of certain topological properties
 of the weak topology of Banach spaces.  It is well known that a Banach space $E$ in
the weak topology is metrizable if and only if $E$ is finite-dimensional. On the other hand,
various topological properties generalizing metrizability have been studied intensively by
 topologists and analysts. Among the others let us mention the
 first countability,  Fr\'{e}chet--Urysohn property, sequentiality, $k$-space property, and countable
tightness (see \cite{Eng,kak}). It is well known that
\[
\xymatrix{
\mbox{metric} \ar@{=>}[r] & \mbox{first}  \atop \mbox{countably} \ar@{=>}[r] & \mbox{Fr\'{e}chet--} \atop \mbox{Urysohn} \ar@{=>}[r] & \mbox{sequential} \ar@{=>}[r] & }
\left\{
\begin{split}
& \mbox{$k$-space} \\
& {\mbox{countable}\mbox{ tight}}
\end{split} \right. ,
\]
and none of these implications can be reversed. Recall the following classical
\begin{theorem}[Kaplansky] \label{kaplansky}
If $E$ is a metrizable lcs,  $E_w$ has countable tight.
\end{theorem}
The Kaplansky theorem can be strengthened by using other sequential concepts which are of great importance for the study of function spaces (see \cite{Arhangel,KocScheep}).

Following Arhangel'skii \cite[II.2]{Arhangel}, we say  that a topological space $X$ has {\em countable fan tightness at a point } $x\in X$ if for each sets $A_n\subset X$, $n\in\NN$, with $x\in \bigcap_{n\in\NN} \overline{A_n}$ there are finite sets $F_n\subset A_n$, $n\in\NN$, such that $x\in \overline{\cup_{n\in\NN} F_n}$; $X$ has {\em countable fan tightness} if $X$ has countable fan tightness at each point $x\in X$. Clearly, if $X$ has countable fan tightness, then $X$ also has countable tightness.

Pytkeev \cite{Pyt} proved that every sequential space satisfies the following property, now known as the  Pytkeev property, which is stronger than
having countable tightness: A topological space $X$ has the {\em Pytkeev property} if for any sets $A\subset X$ and each $x\in \overline{A}\setminus A$, there are infinite subsets $A_1, A_2, \dots $ of $A$ such that each neighborhood of $x$ contains some $A_n$.

A topological space $X$ has the {\em Reznichenko property} (or is a {\em weakly Fr\'{e}chet--Urysohn} space) if $x\in \overline{A}\setminus A$ and $A\subset X$ imply the existence of a countable infinite disjoint family $\Nn$ of finite subsets of $A$ such that for every neighborhood  $U$ of $x$ the family $\{ N\in\Nn : N\cap U=\emptyset\}$ is finite. It is known that
\[
\xymatrix{
\mbox{sequential} \ar@{=>}[r] & \mbox{Pytkeev} \ar@{=>}[r] &  \mbox{Reznichenko}  \ar@{=>}[r] & {\mbox{countable}\atop\mbox{ tight}} },
\]
and none of these implications is reversible (see \cite{malyhin,Pyt}).

For a Tychonoff topological space $X$ we denote by $C_c(X)$ and $C_p(X)$ the space of all continuous real-valued functions on $X$ endowed with the compact-open topology and the topology of pointwise convergence, respectively.  If $X$ is a $\sigma$-compact space, then $C_{p}(X)$ has countable fan tightness by  \cite[II.2.2]{Arhangel} and has the Reznichenko property by \cite[Theorem 19]{KocScheep}. If $E$ is a metrizable lcs, then $X:=(E',\sigma(E',E))$ is $\sigma$-compact by the Alaoglu--Bourbaki theorem. Since $E_w$ embeds into $C_p(X)$, we notice  the following generalization of the Kaplansky theorem.
\begin{theorem}\label{t:Weak-Fan-tight}
Let $E$ be a metrizable lcs (in particular, a Banach space). Then $E_w$ has countable fan tightness and the Reznichenko property.
\end{theorem}

On the other hand,  infinite dimensional  Banach  spaces in the weak topology are never  $k$-spaces.
\begin{theorem}[\cite{S-W}] \label{t:Weak-k-space}
If $E$ is a Banach space, then $E_w$ is a $k$-space (in particular, sequential) if and only if $E$ is finite dimensional.
\end{theorem}

We prove another result of this type having in mind that the concepts of being a $k$-space and having the Pytkeev property are independent in general.
\begin{theorem} \label{t:Weak-Pytkeev}
If $E$ is a normed space, then $E_w$ has the Pytkeev property if and only if $E$ is finite dimensional.
\end{theorem}

These results show that the question when a Banach space endowed with the weak topology is homeomorphic to a certain fixed model space from the infinite-dimensional topology is very restrictive and motivated specialists  to detect the aforementioned properties only for some natural classes of subsets of $E$, e.g., balls or bounded subsets of $E$.
By $B_w$ we denote the closed unit ball $B$ of a Banach space $E$ endowed with the weak topology. It is well known that $B_w$ is (separable) metrizable if and only if the dual space $E'$ is norm separable. Schl\"{u}chtermann and Wheeler obtained in \cite[Theorem 5.1]{S-W} the following characterization.
\begin{theorem}[\cite{S-W}] \label{t:S-W-1}
The following conditions on a Banach space $E$ are equivalent: {\em (a)} $B_w$ is Fr\'{e}chet--Urysohn; {\em (b)} $B_w$ is sequential; {\em (c)} $B_w$ is a $k$-space; {\em (d)} $E$ contains no isomorphic copy of $\ell_1$.
\end{theorem}

By Theorem \ref{t:Weak-Fan-tight}, for every Banach space $E$, $B_w$ has the Reznichenko property. We prove the following result which supplements Theorem \ref{t:S-W-1}.
\begin{theorem} \label{t:Ball-Pytkeev}
For a Banach space $E$, $B_w$ has the Pytkeev property if and only if $E$ contains no isomorphic copy of $\ell_1$.
\end{theorem}
So, if $E$ is the James Tree space, then $B_w$ is Fr\'{e}chet-Urysohn but is not metrizable. On the other hand,  $B_w$ has countable (fan) tightness for every Banach space $E$. This fact and Theorems \ref{t:S-W-1} and \ref{t:Ball-Pytkeev} motivate us to consider another natural generalizations of metrizability.

One of the most immediate extensions of the class of separable metrizable spaces is the class of $\aleph_{0}$-spaces introduced by Michael in \cite{Mich}. Following \cite{Mich}, a topological space $X$ is an {\em $\aleph_0$-space} if $X$ possesses a countable $k$-network. A family $\mathcal{N}$ of subsets of $X$ is a {\em $k$-network}  if for any open subset $U\subset X$ and compact subset $K\subset U$ there exists a finite subfamily $\F\subset\mathcal{N}$ such that $K\subset\bigcup\FF\subset U$. Schl\"{u}chtermann and Wheeler obtained   the following  theorem for $B_w$.
\begin{theorem}[\cite{S-W}] \label{t:S-W-2}
The following conditions on a Banach space $E$ are equivalent: {\em (a)} $B_w$ is (separable) metrizable; {\em (b)} $B_w$ is an $\aleph_0$-space and a $k$-space.
\end{theorem}

Having in mind the Nagata-Smirnov metrization theorem, O'Meara \cite{OMe2} introduced the class of $\aleph$-spaces: A topological space $X$ is called an {\em $\aleph$-space} if $X$ is regular and has a $\sigma$-locally finite $k$-network. Any metrizable space $X$ is an \mbox{$\aleph$-space}. A topological space $X$ is an $\aleph_0$-space if and only if $X$ is a Lindel\"{o}f $\aleph$-space (\cite{GK-GMS1}). In \cite{GK-GMS1,GKKM} it is shown that  each $\aleph$-space $X$ has countable $cs^\ast$-character. Recall  from \cite{BZ} that a topological space $X$ has   {\em countable $cs^\ast$-character}  if for each $x\in X$, there exists a countable family $\DD$ of subsets of $X$, such that for each sequence in $X$ converging to $x$ and each neighborhood $U$ of $x$, there is $D\in\DD$ such that $D\subset U$ and $D$ contains infinitely many elements of that sequence. The importance of this concept for the theory of Topological Vector Spaces may be explained by the following result obtained in \cite{GaK}: If $E$ is a Baire topological vector space (tvs) or a $b$-Baire-like lcs, then $E$ is metrizable if and only if $E$ has countable $cs^\ast$-character.

Tsaban and Zdomskyy \cite{boaz} strengthened the Pytkeev property as follows. A  topological space $X$ has  the {\em strong Pytkeev property} if for each $x\in X$, there exists a countable family $\DD$ of subsets of $X$,   such that for each neighborhood $U$ of $x$ and  each \mbox{$A\subset X$} with $x\in \overline{A}\setminus A$, there is $D\in\DD$ such that $D\subset U$ and $D\cap A$ is infinite. Clearly, the strong   Pytkeev property $\Rightarrow $ the Pytkeev property, however in general, the Fr\'{e}chet--Urysohn property $\not\Rightarrow $ the strong   Pytkeev property $\not\Rightarrow $   $k$-space (see \cite{GKL2}). For any Polish  space $X$, the function space $C_c(X)$ has the strong Pytkeev property by \cite{boaz}. This result has been extended to  \v{C}ech-complete spaces $X$, see \cite{GKL2}. Moreover, if $E$ is a strict \mbox{$(LM)$-space}, a sequential dual metric space, or a $(DF)$-space of countable tightness, then $E$ has the strong Pytkeev property \cite{GKL2}. In particular, the space $D'(\Omega)$ of distributions over an open subset $\Omega\subset\mathbb{R}^{n}$ has the strong Pytkeev property while being  not a \mbox{$k$-space}, see \cite{GKL2}.

Being motivated by above facts, we generalize Theorem \ref{t:S-W-2} as follows (the equivalence (i)$\Leftrightarrow$(ii) is probably well known, but hard to locate and we propose  an elementary proof of it below).
\begin{theorem} \label{t:Ball-Metriz}
The following conditions on a Banach space $E$ are equivalent:
\begin{enumerate}
\item[{\rm (i)}] $B_w$ is (separable) metrizable;
\item[{\rm (ii)}]  $B_w$ is first countable;
\item[{\rm (iii)}]  $B_w$ has the strong Pytkeev property;
\item[{\rm (iv)}]  $B_w$ is an $\aleph$-space and a $k$-space;
\item[{\rm (v)}]  $B_w$ has countable $cs^\ast$-character and is a $k$-space;
\item[{\rm (vi)}]  $E_w$  is an $\aleph$-space and $B_w$ is  a $k$-space;
\item[{\rm (vii)}]  $E_w$ has countable $cs^\ast$-character and $B_w$ is  a $k$-space;
\item[{\rm (viii)}]  $E_w$ has countable $cs^\ast$-character and contains no  isomorphic copy of $\ell_1$;
\item[{\rm (ix)}]  $E'$ is separable.
\end{enumerate}
\end{theorem}
Note that the James Tree space and the sequence space $\ell_{1}(\mathbb{R})$ (which is  an \mbox{$\aleph$-space} in the weak topology, see \cite{GKKM}) show that Theorem \ref{t:Ball-Metriz} fails if one of the assumptions on $E$ is dropped.

It is natural to ask about  good conditions forcing metrizable $B_w$ to be completely metrizable.
Edgar and Wheeler proved the following characterization:
\begin{theorem}[\cite{edgar}] \label{t:E-W}
Let $E$ be a separable Banach space. Then the following conditions are equivalent: {\em (1)} $B_w$ is completely metrizable; {\em (2)} $B_w$ is a Polish space; {\em (3)} $E$ has property $(PC)$ and is an Asplund space; {\em (4)} $B_w$ is metrizable, and every closed subset of it is a Baire space.
\end{theorem}

Since a metrizable space $X$ admits a complete metric generating its topology
 if and only if it is \v{C}ech-complete, and any \v{C}ech-complete space is a $k$-space, Theorem \ref{t:Ball-Metriz} immediately implies the following corollary  which supplements the Edgar-Wheeler Theorem \ref{t:E-W}.
\begin{corollary} \label{c:B-Polosh}
The following conditions on a Banach space $E$ are equivalent:
\begin{itemize}
\item[{\rm (i)}] $B_w$ is Polish;
\item[{\rm (ii)}] $B_w$ is a \v{C}ech-complete  $\aleph$-space;
\item[{\rm (iii)}]  $B_w$ has countable $cs^\ast$-character and is  \v{C}ech-complete;
\item[{\rm (iv)}]  $E_w$ has countable $cs^\ast$-character and $B_w$ is  a \v{C}ech-complete space.
\end{itemize}
\end{corollary}
Note that the \v{C}ech-completeness of $B_w$ in (iii) cannot be replaced by the weaker condition that every closed subset of $B_w$ is a Baire space (as in Theorem \ref{t:E-W}(4)) even if $E_w$ is an $\aleph_0$-space, see Example \ref{exa:l1-Aleph0-Baire}.

The paper is organized as follows. Section \ref{sec:1} deals with the (strong) Pytkeev property and the fan tightness. Applying these concepts we prove that any  \mbox{$(DF)$-space} $E$ is normable if and only if $E$ has countable fan tightness,  as well as the strong dual $F$ of a strict $(LF)$-space has a countable fan tightness if and only if $F$ is metrizable. Hence the space $D'(\Omega)$ of distributions over an open subset $\Omega\subset\mathbb{R}^{n}$ does not have countable fan tightness although it has countable tightness.  Theorems \ref{t:Weak-Pytkeev} and \ref{t:Ball-Pytkeev} are proved in Section \ref{sec:11}, and in Section \ref{sec:2} we prove Theorem \ref{t:Ball-Metriz}. Also we provide some examples.

\section{The strong Pytkeev property and the fan tightness in lcs} \label{sec:1}

Below we provide a simple proof of  the following result to keep the paper self-contained.
\begin{proposition}[\cite{ Banakh1}] \label{p:FirstCount}  The following assertions are equivalent for a topological space $X$:
\begin{itemize}
\item[{\rm (i)}] $X$  is first countable.
\item[{\rm (ii)}] $X$ has the strong  Pytkeev property and countable fan tightness.
\end{itemize}
\end{proposition}
\begin{proof}
Let $\mathcal N_0$ be a family witnessing the strong Pytkeev property  at $x\in X$.
We claim that $\mathcal N:=\{\bigcup \mathcal N' : \mathcal N' \in [\mathcal N_0]^{<\mathbb{N}}\}$
is a local base at $x$. If not, there exists open $ U \ni x$ such that no element of $\mathcal N$
contained in $U$ is a neighborhood of $x$. Let $\{N_i:i\in\mathbb{N}\}$ be the enumeration of all elements of
$\mathcal N_0$ which are subsets of $U$. It follows from the above that $x$ lies in the closure
of $B_n:= X\setminus\bigcup_{i\leq n} N_i$ for all $n$, and hence we can select a finite
subset $A_n$ of $B_n$ such that $\bigcup_{n\in\mathbb{N}}A_n$ has $x$ in its closure.
But this is a contradiction because obviously no $N_i$ can have infinite intersection with
$\bigcup_{n\in\mathbb{N}}A_n$.
\end{proof}

In \cite[Question 5]{GKL2} we ask whether there exists a lcs $E$ such that its dual $E'$ has uncountable algebraic dimension and $E_w$ has the strong Pytkeev property.  Theorem \ref{t:Weak-Fan-tight} and Proposition \ref{p:FirstCount} immediately imply a negative answer to this question for any metrizable lcs $E$.
\begin{corollary} \label{c:Pytkeev-Metriz}
Let $E$ be a metrizable lcs. Then $E_w$ has the strong Pytkeev property if and only if $E$ is finite-dimensional.
\end{corollary}

In \cite{GKL2} we provided large classes of lcs having the strong Pytkeev property including an important class of $(DF)$-spaces. Recall that a lcs $E$ is a $(DF)$-space if $E$ has a fundamental sequence of bounded sets and every bounded set in the strong dual $(E',\beta(E',E))$ of $E$ which is the countable union of equicontinuous sets is itself equicontinuous. The strong  dual $(E',\beta(E',E))$ of a metrizable lcs is a $(DF)$-space, see   \cite[Theorem 8.3.8]{bonet}. Below we use the following result.
\begin{fact}[\cite{GKL2}] \label{quasi}
A $(DF)$-space $E$ has countable tightness if and only if $E$ has the strong Pytkeev property.
\end{fact}

Since a $(DF)$-space is normable if and only if it is metrizable, Fact \ref{quasi} and Proposition \ref{p:FirstCount} imply
\begin{corollary}\label{fre}
A $(DF)$-space $E$ is normable if and only if $E$ has countable fan tightness.
\end{corollary}

We use Fact \ref{quasi} also to prove the following result which supplements \cite[Theorem 3(iii)]{GKL2}.
\begin{proposition}\label{p:dist}
Let $E$ be a strict $(LF)$-space and $E':=(E',\beta(E',E))$ its strong dual. Then $E'$ has strong Pytkeev property if and only if $E'$ has countable tightness.
\end{proposition}
\begin{proof}
Assume that $E'$ has countable tightness.  Let $E$ be a strict inductive limit of a sequence $\{ E_{n}\}_{n}$ of Fr\'{e}chet spaces. For each $n\in\mathbb{N}$, the strong dual $(E'_n)_{\beta}$ of $E_{n}$ is a complete $(DF)$-space. Since $E$ is the strict inductive limit, the space $E'_{\beta}$ is linearly homeomorphic with the \textit{projective
limit} of the sequence $\{(E'_{n})_{\beta}\}_{n}$ of $(DF)$-spaces, and moreover,  $E'_{\beta}$ can be continuously mapped onto each $(E'_{n})_{\beta}$ by an open mapping, see \cite{horvath}. Since $E'$ has countable tightness, for every $n\in\NN$, the quotient space $(E'_{n})_{\beta}$ has countable tightness by \cite[Proposition 3]{Arch-Pon}. Hence all  spaces $(E'_{n})_{\beta}$ have the strong Pytkeev property by Fact \ref{quasi}. Therefore the product $\prod_{n}(E'_n)_{\beta}$ also has the strong Pytkeev property by \cite{GK-GMS1}. Consequently  $E'$ has the strong Pytkeev property.
\end{proof}
Propositions \ref{p:FirstCount}  and \ref{p:dist} imply
\begin{corollary}
The strong dual $F$ of a strict $(LF)$-space has  countable fan tightness if and only if $F$ is metrizable.
\end{corollary}

We end this section with the next two examples.
\begin{example} {\em
If $\Omega\subset\mathbb{R}^{n}$ is  an open set, then the space of test functions $\mathfrak{D}(\Omega)$ is a complete Montel $(LF)$-space. As usually, $\mathfrak{D}'(\Omega)$ denotes its strong dual,  the space of distributions. We proved in \cite{GKL2} that $\mathfrak{D}'(\Omega)$ has the strong Pytkeev property, but it is not a $k$-space. Now Proposition \ref{p:FirstCount} implies that $\mathfrak{D}'(\Omega)$  does not have countable fan tightness. }
\end{example}

\begin{example} {\em
Let $E:=C_{c}(X)$ be any Fr\'echet lcs which is not normed (for example $E:=C_{c}(\mathbb{R})$). Then $(E',\beta(E',E))$ does not have countable fan tightness although it has countable tightness. Indeed, since $C_c(X)$  is quasi-normable, it is  distinguished by \cite[Corollary 1]{bierstedt2}, and hence  the space $E'_\beta := (E',\beta(E',E))$ is quasibarrelled. Then, by  \cite[Theorem 12.3]{kak}, the space $E'_\beta$ has countable tightness. Finally, Corollary \ref{fre} implies that $E'_\beta$ does not have countable fan tightness.}
\end{example}

\section{Proofs of Theorems \ref{t:Weak-Pytkeev} and  \ref{t:Ball-Pytkeev}} \label{sec:11}

Denote by $S_{E}$ the unit sphere  of a normed space $E$.

\begin{lemma} \label{lemma:Funct-Seq}
Let $\{ A_n\}_{n\in\NN}$ be a sequence of unbounded subsets of a normed space $E$. Then there is $\chi\in S_E$ such that
\[
A_n\setminus \{ x\in E: |\langle\chi, x\rangle| <1/2\} \not=\emptyset, \mbox{ for every $n\in\NN$.}
\]
\end{lemma}
\begin{proof}
First we prove the following claim.

{\bf Claim}. {\em  There is a  sequence $S=\{ a_n\}_{n\in \NN}\subset E$, where $a_n\in A_n$ for every $n\in \NN$,  and a sequence $T=\{ \chi_n\}_{n\in\NN} \subset S_{E'}$ such that}
\begin{equation} \label{equation-1}
|\langle \chi_n, a_i\rangle| >1, \; \mbox{ for every } 1\leq i\leq n <\infty.
\end{equation}

We build $S$ and $T$ by induction. For $n=1$ take arbitrarily $a_1\in A_1$ such that $\| a_1\| >1$ (note that $A_1$ is unbounded) and $\chi_1\in S_{E'}$ such that $|\langle \chi_1, a_1\rangle|=\| a_1\|$. Assume that for $n\in\NN$, we found  $a_i\in A_i$ and $\chi_k\in  S_{E'}$ such that
\begin{equation} \label{equation-2}
|\langle \chi_k, a_i\rangle| >1, \; \mbox{ for every } 1\leq i\leq k\leq n.
\end{equation}
We distinguish between two cases.

{\em Case 1}. The set $\langle \chi_n, A_{n+1}\rangle$ is an unbounded subset of $\IR$. Then we choose $a_{n+1}\in A_{n+1}$ such that $|\langle \chi_n, a_{n+1}\rangle|>1$ and set $\chi_{n+1}:=\chi_n$. It is clear that (\ref{equation-2}) holds also for $n+1$.

{\em Case 2}.  $N:= \sup\{ |\langle \chi_n, a\rangle |: a\in A_{n+1}\} <\infty$.
Since $A_{n+1}$ is unbounded, there is $\eta\in S_E$ such that $\langle \eta, A_{n+1}\rangle$ is an unbounded subset of $\IR$. Choose $0<\lambda <1$ such that
\begin{equation}  \label{equation-3}
\lambda |\langle \chi_n, a_i\rangle| - (1-\lambda)|\langle \eta, a_i\rangle| >1, \; \mbox{ for every } 1\leq i\leq n,
\end{equation}
and choose $a_{n+1}\in A_{n+1}$ such that
\begin{equation}  \label{equation-4}
(1-\lambda)|\langle \eta, a_{n+1}\rangle| >N+1.
\end{equation}
Set $\xi := \lambda \chi_n + (1-\lambda)\eta$. Then by (\ref{equation-3}), we have
\begin{equation}  \label{equation-5}
|\langle \xi, a_i\rangle|\geq \lambda |\langle \chi_n, a_i\rangle| - (1-\lambda)|\langle \eta, a_i\rangle| >1, \; \mbox{ for every } 1\leq i\leq n,
\end{equation}
and, by (\ref{equation-4}),
\begin{equation}  \label{equation-6}
|\langle \xi, a_{n+1}\rangle|\geq (1-\lambda)|\langle \eta, a_{n+1}\rangle| - \lambda |\langle \chi_n, a_{n+1}\rangle| >1.
\end{equation}
Finally we set $\chi_{n+1} := \xi/\|\xi\|\in S_{E'}$. Since  $\| \xi\|\leq 1$, (\ref{equation-5}) and (\ref{equation-6}) imply that $\chi_{n+1}$ and $a_{n+1}$ satisfy (\ref{equation-2}). The claim is proved.

Since the unit ball $B_{E'}$ of the dual $E'$ is compact in the weak$^\ast$ topology, we can find a cluster point $\chi \in B_{E'}$  of the sequence $T$ defined in Claim.  In particular, for every $i\in\NN$ there is $n>i$ such that $|\langle \chi_n, a_i\rangle - \langle \chi, a_i\rangle| <1/2$. Then (\ref{equation-1}) implies
\begin{equation} \label{equation-7}
|\langle \chi, a_i\rangle| >1/2 , \; \mbox{ for every } i\in\NN.
\end{equation}
Now (\ref{equation-7}) implies that  $a_n \in A_n \setminus \{ x\in E: |\langle \chi, x\rangle|< 1/2\}$ for every $n\in\NN$, which  proves the lemma.
\end{proof}

We are at position to prove Theorem \ref{t:Weak-Pytkeev}.

\begin{proof} [Proof of Theorem \ref{t:Weak-Pytkeev}]
Assume towards a contradiction that there is an infinite dimensional normed space $E$ such that $E_w$ has the Pytkeev property.

{\bf Claim}. {\em For every subset $A\subset E_w$ with $0\in\overline{A}\setminus A$, there is a bounded subset $D$ of $A$ such that $0\in\overline{D}$.}

Indeed, suppose that there is a subset $A$ of $E_w$ with $0\in\overline{A}\setminus A$ and such that $0\not\in \overline{A\cap nB}$ for every $n\in\NN$. So  $0\in \bigcap_{n\in\NN} \overline{A\setminus nB}$.  Since $E_w$ has countable fan tightness (see Theorem \ref{t:Weak-Fan-tight}), there are finite subsets $F_n \subset A\setminus nB$ such that $0\in \overline{\cup_{n\in\NN} F_n}$. Set $F:=\bigcup_{n\in\NN} F_n$. Fix arbitrarily sequence $A_1, A_2, \dots $ of  infinite subsets of $F$.

By the construction of $F$, all sets $A_n$ are unbounded. Lemma \ref{lemma:Funct-Seq} implies that  there is a weakly open neighborhood $U$ of zero such that $A_n\setminus U$ is not empty for every $n\in\NN$. Thus $E_w$ does not have the Pytkeev property. This contradiction proves the claim.

Now Theorem 14.3 of \cite{kak} and Claim  imply that $E_w$ is a Fr\'{e}chet-Urysohn space, and hence, by Theorem \ref{t:Weak-k-space}, $E$ is finite dimensional. This contradiction shows that $E_w$ does not have the Pytkeev property.
\end{proof}

To prove Theorem \ref{t:Ball-Pytkeev} we define the following subset of $S_{\ell_1}$
\begin{equation}\label{equ:L1-A}
\begin{split}
A:=\{ & a=\big(a(i)\big)\in S_{\ell_1} | \\
 & \exists m,n\in\NN : a(m)=-a(n)=1/2,  \mbox{ and }  a(i) =0 \mbox{ otherwise} \}.
\end{split}
\end{equation}
\begin{lemma}
$0\in \overline{A}$.
\end{lemma}

\begin{proof}
Let $U$ be a neighborhood of $0$ of the canonical form
\[
U=\left\{ x\in\ell_1 : |\langle\chi_k,x\rangle|<\epsilon , \mbox{ where } \chi_k\in S_{\ell_\infty} \mbox{ for }  1\leq k\leq s\right\}.
\]
Let $I$ be  an infinite subset of $\NN$ such that, for every $1\leq k\leq s$, either $\chi_k(i)>0$  for all $i\in I$, or $\chi_k(i)=0$  for all $i\in I$, or $\chi_k(i)<0$  for all $i\in I$. Take a natural number $N> 1/\epsilon$. Since $I$ is infinite, by induction,  one can find $m,n\in\NN$ satisfying the following condition:  for every $1\leq k\leq s$ there is $0< t_k\leq N$ such that
\begin{equation}\label{equ:L1-1}
\frac{t_k -1}{N} \leq \min\big\{ |\chi_k(m)| , |\chi_k(n)| \big\} \leq \max\big\{ |\chi_k(m)| , |\chi_k(n)| \big\} \leq \frac{t_k}{N}.
\end{equation}
Set $a=\big(a(i)\big)\in A$, where $a(m)=-a(n)=1/2$, and $a(i)=0$ otherwise. Then, by the construction of $I$ and (\ref{equ:L1-1}), we obtain  $|\chi_k(a)|<1/N <\epsilon$ for every $1\leq k\leq s$. Thus $a\in U$, and hence $0\in \overline{A}$.
\end{proof}

\begin{proposition}\label{p:Pytkeev-Ball}
The unit ball $B_w$ of  $\ell_1$ in the weak topology does not have the Pytkeev property.
\end{proposition}

\begin{proof}
It is enough to show that for any sequence $A_1,  A_2,\dots$ of  infinite subsets of the set $A$ defined in (\ref{equ:L1-A}) there is a neighborhood of $0$ which does not contain $A_i$ for every $i\in\NN$.

Fix arbitrarily $a_1\in A_1$. So $a_1(m_1)=-a_1(n_1)=1/2$ for some $m_1,n_1\in\NN$. Set $\chi(m_1)=-\chi(n_1)=1$ and $D_1=\supp(a_1)=\{m_1,n_1\}$. Then
\[
\left| \sum_{i\in D_1} \chi(i) a_1(i)\right| = |\chi(m_1)a_1(m_1) +\chi(n_1)a_1(n_1)|=1.
\]
Since $A_2$ is infinite, we can choose $a_2\in A_2$ such that  \mbox{$\supp(a_2)=\{m_2,n_2\}\not\subset D_1$}. If $\supp(a_2)\cap D_1=\emptyset$, we set $\chi(m_2)=-\chi(n_2)=1$. If $m_2\in D_1$, we set $\chi(n_2)=-\chi(m_2)$, and if $n_2\in D_1$, we set $\chi(m_2)=-\chi(n_2)$.  Then
\[
\left| \sum_{i\in D_2} \chi(i) a_2(i)\right| = |\chi(m_2)a_2(m_2) +\chi(n_2)a_2(n_2)|=1.
\]
 Put $D_2=\supp(a_2)\cup D_1$. Suppose we found $a_k\in A_k$ for $1\leq k\leq s$ and defined $D_s =\cup_{k\leq s} \supp(a_k)$ and $\chi(i)=\pm 1$ for $i\in D_s$ such that
\begin{equation}\label{equ:L1-2}
\left| \sum_{i\in D_s} \chi(i)a_k(i) \right| =1 , \; \mbox{ for every } 1\leq k\leq s.
\end{equation}
Since $A_{s+1}$ is infinite and $D_s$ is finite, we can choose $a_{s+1}\in A_{s+1}$ such that
\[
\supp(a_{s+1})=\{m_{s+1},n_{s+1}\}\not\subset D_s.
\]
If $\supp(a_{s+1})\cap D_s=\emptyset$, we set $\chi(m_{s+1})=-\chi(n_{s+1})=1$. If $m_{s+1}\in D_s$, we set $\chi(n_{s+1})=-\chi(m_{s+1})$, and if $n_{s+1}\in D_s$, we set $\chi(m_{s+1})=-\chi(n_{s+1})$.  So $|\chi(m_{s+1})a_{s+1}(m_{s+1}) +\chi(n_{s+1})a_{s+1}(n_{s+1})|=1$. Put $D_{s+1}=\supp(a_{s+1})\cup D_s$. In particular, (\ref{equ:L1-2}) holds also for $s+1$.
Put $D=\cup_{s\in\NN} D_s$.

Set $\chi=(\chi(i))\in S_{\ell_\infty}$, where $\chi(i)=\chi(n_s)$ if $i=n_s$,  $\chi(i)=\chi(m_s)$ if $i=m_s$ for some $s\in\NN$, and $\chi(i)=0$ if $i\not\in D$. Then  (\ref{equ:L1-2}) implies that $|\chi(a_s)|=1$ for every $s\in\NN$. Finally we set $U=\{ x\in \ell_1 : |\langle\chi,x\rangle|<1/2\}$. Then $a_s \in A_s \setminus U$ for every $s\in\NN$. Thus $B_w$ does not have the Pytkeev property.
\end{proof}

Now we are ready to prove Theorem \ref{t:Ball-Pytkeev}.

\begin{proof}[Proof of Theorem \ref{t:Ball-Pytkeev}]
If $B_w$ has the Pytkeev property, then $E$ contains no isomorphic copy of $\ell_1$ by Proposition \ref{p:Pytkeev-Ball}. The converse assertion follows from Theorem \ref{t:S-W-1}.
\end{proof}

\section{Proof of Theorem \ref{t:Ball-Metriz}} \label{sec:2}

We need the following fact  which is similar to Proposition 7.7 of \cite{Mich}.
\begin{proposition} \label{p:Sequence-Aleph}
Assume that a regular space $X$ is covered by an increasing sequence $\{ A_n\}_{n\in\NN}$ of closed subsets.
\begin{enumerate}
\item[{\rm (i)}] If all $A_n$ are $\aleph$-spaces (resp. an $\aleph_0$-spaces), and if each compact $K\subset X$ is covered by some $A_n$, then $X$ is an $\aleph$-space (resp. an $\aleph_0$-space, respectively).
\item[{\rm (ii)}] If all $A_n$ have countable  $cs^\ast$-character and  each convergent sequence $S\subset X$ is covered by some $A_n$, then $X$ has countable  $cs^\ast$-character.
\end{enumerate}
\end{proposition}
\begin{proof}
(i): Assume that all $A_n$ are $\aleph$-spaces and let $\DD_n =\bigcup_{k\in\NN} \DD_{n,k}$ be a $\sigma$-locally finite $k$-network for $A_k$. It is easy to see that the family
$
\DD := \bigcup_{n\in\NN}\bigcup_{k\in\NN} \DD_{n,k}
$
is $\sigma$-locally finite in $X$ and covers $X$. If $K\subset U$ with $K$ compact and $U$ open in $X$, take $A_n$ such that $K\subset A_n$ and choose a finite subfamily $\FF \subset \DD_n \subset \DD$ such that $K\subset \bigcup\FF \subset A_n \cap U \subset U$. This means that $\DD$ is a $k$-network for $X$, and hence $X$ is an $\aleph$-space. The case $A_n$ are $\aleph_0$-spaces is proved analogously.

(ii): Fix $x\in X$, an open neighborhood $U$ of $x$ and a sequence $S=\{ s_n\}_{n\in \NN}$ converging to $x$. Set $I(x):=\{ n\in\NN: x\in A_n\}$ and let $\DD_n (x)$ be a countable $cs^\ast$-network at $x$ in $A_n, n\in I(x)$. Set $\DD(x) =\bigcup \{ \DD_n(x): n\in I(x)\}$. Take $n\in I(x)$ such that $S\subset A_n$. Then, by definition, we can find $D\in \DD_n(x) \subset \DD(x)$ such that $x\in D\subset U$ and $\{ n\in \NN : s_n\in D\}$ is infinite. Thus $\DD(x)$ is a countable $cs^\ast$-network at $x$.
\end{proof}

\begin{corollary}
Let $A$ be a closed subset of a tvs $E$.
\begin{enumerate}
\item[{\rm (i)}] If every compact subset $K\subset E$  is contained in some $nA$, then $E$ is an $\aleph$-space (resp. an $\aleph_0$-space) if and only if so is $A$.
\item[{\rm (ii)}] If every convergent sequence $S\subset E$ is contained in some $nA$, then $E$ has countable $cs^\ast$-character  if and only if  so does $A$.
\end{enumerate}
\end{corollary}
\begin{proof}
Clearly $E=\bigcup_{n\in\NN} nA$. Since $A$ and $nA$ are topologically isomorphic for every $n\in\NN$, the assertion follows from Proposition \ref{p:Sequence-Aleph}.
\end{proof}
Since any weakly compact subset of a Banach space is contained in $nB$ for some $n\in\NN$, 
the last corollary implies (cf. \cite[Theorem 4.2]{S-W})
\begin{corollary} \label{c:Sequence-Aleph}
\begin{enumerate}
\item[{\rm (i)}] A Banach space $E$ is a weakly $\aleph$-space (resp. a weakly $\aleph_0$-space) if and only if so is $B_w$.
\item[{\rm (ii)}] A Banach space $E$ has countable $cs^\ast$-character  if and only if  so does  $B_w$.
\end{enumerate}
\end{corollary}

Any first countable topological group is metrizable, but there are first-countable non-metrizable compact  spaces. However, for the unit ball of Banach spaces we have the following result which supplements Theorem \ref{t:S-W-1} (probably  known but hard to locate).
\begin{proposition}\label{p:Ball-FirstCountable}
Let $E$ be a Banach space. Then $B_w$ is metrizable if and only if it is first countable.
\end{proposition}
\begin{proof}
Assume that $B_w$ is first countable. Denote by $\mathcal{W}$ the standard group  uniformity
\[
\big\{ (x,y)\in E^2:x-y\in U,\, \mbox{ where } 0\in U\in \sigma(E,E') \big\}
\]
on $E$ generating the topology $\sigma(E,E')$. Since $B_w$ is homeomorphic to the ball $2B_w =B_w -B_w$ of radius $2$, the set $2B_w$ is also first countable. So there exists a sequence $\{ U_n\}_{n\in\NN}\subset \sigma(E,E')$ of open symmetric neighborhoods of $0$, such that $\{U_n \cap 2B_w :n\in\mathbb{N}\}$ is a countable base  at $0$. Let us show that the family $\{W_n:n\in\mathbb{N}\}$, where $W_n=\{(x,y)\in B\times B : x-y\in U_n\}$, is a countable base of the uniformity $\mathcal{W}|_B$.
Indeed, fix arbitrarily $W=\{(x,y)\in B\times B : x-y\in U\} \in \mathcal{W}|_B$ with $0\in U\in \sigma(E,E')$. Choose $n\in\NN$ such that $U_n\cap 2B \subset U\cap 2B$. Then, for every $(x,y)\in W_n$, we have $x-y\in U_n$ and also $x-y\in 2B$. So $x-y\in U_n\cap 2B \subset U\cap 2B$, and hence $(x,y)\in W$.

By Theorem 8.1.21 of \cite{Eng}, the uniformity $\mathcal{W}|_B$ is generated by a metric $\rho$ on $B$. So the topologies $\tau_\rho$ and $\tau_B$ on $B$ induced by the metric $\rho$ and $\mathcal{W}|_B$ coincide (see \cite[4.1.11, 4.2.6 and 8.1.18]{Eng}). Taking into account that $\tau_B$ and $\sigma(E,E')|_B$ also coincide (see \cite[8.1.17]{Eng}), we obtain that $B_w$ is metrizable.  The converse assertion is trivial.
\end{proof}

Recall that a topological space $X$ has the \textit{property} $\left( \alpha_{4}\right) $ if for any $\{x_{m,n}:\left( m,n\right) \in \mathbb{N}\times\mathbb{N}\}\subset X$
 with $\lim_{n}x_{m,n}=x\in X,$ $m\in \mathbb{N},$ there exists a sequence $\left( m_{k}\right) _{k}$ of distinct natural numbers and a sequence $\left( n_{k}\right) _{k}$ of natural numbers such that $\lim_{k}x_{m_{k},n_{k}}=x$.

 The following result is similar to    \cite[Lemma 3.2]{GKKLP} and is essentially a corollary of the proof of \cite[Theorem~4]{nyikos}. However,  the latter is devoted  to  general (i.e., not necessary Hausdorff)  topological groups which makes  big parts of its proof  not relevant for us, and hence we believe that a streamlined proof of the proposition below is still of some value.
\begin{proposition} \label{important}
Any  convex  and Fr\'{e}chet-Urysohn subset $F$
 of a topological vector space (tvs) $E$ has the property $(\alpha_4)$.
\end{proposition}
\begin{proof}
Let  $(x_{m,n})_{n\in\NN}$ be a sequence of elements of $F$ convergent to $x\in F$ for all $m\in \NN$. Denote by $J$ the set of all $m\in\NN$ for which the family $\{ n\in\NN : x=x_{m,n}\}$ is nonempty. If $J$ is infinite the assertion  is trivial. If $J$ is finite, without loss of generality  we may additionally assume that $x\neq x_{m,n}$ for all $m,n\in\NN$. Two cases are possible.

(i): The set $I:=\left\{m\in\NN:\exists n(m)\in\NN\: \big(x=1/2(x_{m,n(m)}+x_{1,m})\big)\right\}$ is infinite. Since $x_{m,n(m)}=2x-x_{1,m}$ for all $m\in I$ and $(x_{1,m})_{m\in I}$ converges to $x$, so does the sequence $(x_{m,n(m)})_{m\in I}$.

(ii): The set $I$ defined above is finite, say  $\max(I)=q$. Set
\[
X=\{1/2(x_{m,n}+x_{1,m}) :  n\in \NN, m>q \},
\]
and observe that $X\subset F\setminus\{x\}$ and $x\in\overline{X}$. Since $F$ has the Fr\'{e}chet--Urysohn property, there exists a sequence
$\big(1/2(x_{m_k,n_k}+x_{1,m_k})\big)_{k\in\NN}$ of elements of $X$ converging to $x$. It follows from the assumption $x\not\in\{ x_{1,m}:m\in\NN\}$ that for every $m\in\NN$ there are at most finitely many $k$ such that $m_k=m$, and therefore passing to a subsequence if necessary we may assume that $m_{k+1}>m_k$ for all $k$.
Thus $\big(x_{1,m_k}\big)_{k\in\NN}$ converges to $x$, and hence so does  $\big(x_{m_k,n_k}\big)_{k\in\NN}$.

In any of these cases there exists a sequence which converges to $x$ and meets $(x_{m,n})_{n\in\NN}$ for infinitely many $m$, which completes our proof.
\end{proof}

\begin{corollary}\label{aleph}
A convex  subset $D$ of a tvs $E$ is metrizable if and only if $D$ is both a Fr\'echet-Urysohn space and an $\aleph$-space.
\end{corollary}
\begin{proof}
This follows from Proposition \ref{important} and the fact that a topological space which is an $\aleph$-space and Fr\'echet-Urysohn with property $(\alpha_{4})$ is metrizable, see \cite[Theorem 2.2]{GKKM}.
\end{proof}
Next proposition is crucial for the proof of Theorem \ref{t:Ball-Metriz}.
\begin{proposition} \label{p:ball-Metriz}
Let $E$ be a Banach space. If $B_w$  is a $k$-space and has countable $cs^\ast$-character, then it is metrizable.
\end{proposition}
\begin{proof}
In the proof all subspaces $X$ of $E$ are considered with the weak topology $\sigma(E,E')|_X$. By Theorem \ref{t:S-W-1} we can assume that $B_w$ is  Fr\'{e}chet-Urysohn. By Proposition \ref{p:Ball-FirstCountable}, it is enough to prove that $B_w$ is first countable.

Fix $x\in B$, a neighborhood  $U$  of $x$ in $B_w$ and   a countable $cs^\ast$-network $\mathcal{N}$ at $x$ in $B_w$. Let $\{N_i:i\in\mathbb{N}\}$ be an enumeration of all elements of $\mathcal{N}$ which are subsets of $U$.  We claim that there exists $m$ such that $\bigcup_{i\leq m} N_i$ is a neighborhood of $x$ in $B_w$. If not, for every $m\in\NN$ there exists a sequence $\{x_{m,n}:n\in\mathbb{N}\}\subset B\setminus \bigcup_{i\leq m} N_i$ converging to $x$. Since $B_w$ is a bounded subset of $E_w$ we apply  Proposition  \ref{important} to find a sequence $\left( m_{k}\right)_{k}$ of distinct natural numbers and a sequence $\left( n_{k}\right)_{k}$ of natural numbers such that $\lim_{k}x_{m_{k},n_{k}}=x$. It follows that there exists $i\in\mathbb{N}$ such that the intersection $C:=N_i\cap\{x_{m_{k},n_{k}}:k\in\mathbb{N}\}$ is infinite. On the other hand, $x_{m_{k},n_{k}}\not\in N_i$ for all $k> i$  because $m_k> i$, and hence $C$ is finite. This contradiction shows that $\{\bigcup\mathcal N':\mathcal N'\in [\mathcal N]^{<\mathbb{N}}\}$ is a countable base at $x$ in $B_w$.
\end{proof}

Now we prove Theorem \ref{t:Ball-Metriz}. 

\begin{proof}[Proof of Theorem \ref{t:Ball-Metriz}]
(i)$\Rightarrow$(ii)$\Rightarrow$(iii) is clear. (iii)$\Leftrightarrow$(i) follows from
Theorem \ref{t:Weak-Fan-tight}  and Propositions  \ref{p:FirstCount} and \ref{p:Ball-FirstCountable}. The separable case and (i)$\Leftrightarrow$(ix) are well known, see \cite{fabian}. (vii)$\Leftrightarrow$(viii) follows from Theorem \ref{t:S-W-1}.  (i)$\Rightarrow$(iv) is clear, and (iv)$\Rightarrow$(v)  and (vi)$\Rightarrow$(vii) follow from the fact that every $\aleph$-space has countable $cs^\ast$-character (see \cite[Corollary 3.8]{GK-GMS1}). (iv)$\Leftrightarrow$(vi) and (v)$\Leftrightarrow$(vii) follow from Corollary \ref{c:Sequence-Aleph}. Finally, (v)$\Leftrightarrow$(i) follows from Proposition \ref{p:ball-Metriz}.
\end{proof}
The following corollary generalizes Corollary 5.6 of \cite{GKKM}.
\begin{corollary}
Let $E$ be a  Banach space not containing $\ell_1$. Then $E'$ is separable if and only if  $E_w$ has countable $cs^\ast$-character.
\end{corollary}

In \cite[Corollary 5.3]{GKKM} we proved that a reflexive Fr\'{e}chet space is a weakly $\aleph$-space if and only if $E'$ is separable. For reflexive Banach spaces we strengthen this result as follows.
\begin{corollary} \label{c:Reflex-Separ}
Let $E$ be an infinite dimension reflexive Banach space. Then $E_w$ does not have the
strong Pytkeev property. Moreover, the following conditions are equivalent: (i) $E_w$ has countable $cs^\ast$-character; (ii) $B_w$ has countable $cs^\ast$-character; (iii) $E'$ is separable;
(iv) $B_w$ is Polish.
\end{corollary}
\begin{proof}
The space $E_w$ does not have the strong Pytkeev property  by Corollary \ref{c:Pytkeev-Metriz}.
By Theorem B of \cite{edgar},  $B_w$ is \v{C}ech-complete, and
thus the second assertion follows from Theorem \ref{t:Ball-Metriz}
and the fact that any metrizable separable \v{Cech}-complete space is Polish. 
\end{proof}
For example, for $1<p<\infty$, the reflexive space $\ell_p(\Gamma)$ has countable $cs^\ast$-character in the weak topology if and only if $\Gamma$ countable, which  extends Example 3.1 of \cite{GKKM}.
Corollary \ref{c:Reflex-Separ} also shows  that for the strong Pytkeev property there does not exist a result analogous to Corollary  \ref{c:Sequence-Aleph}. Indeed, if $E=\ell_2$, then $B_w$ has the strong Pytkeev property because it is Polish, but $E_w$ does not have the strong Pytkeev property.

Below we consider examples clarifying relations between the notions from Theorems \ref{t:Ball-Metriz} and \ref{t:E-W} and  other notions considered in \cite{edgar}.

\begin{example} \label{exa:l1-Aleph0-Baire}{\em
Let $E=\ell_1$. Then $E$ is a weakly $\aleph_0$-space by \cite[7.10]{Mich} and $B_w$ is almost \v{C}ech-complete \cite[6(9)]{edgar}. So every closed subset of $B_w$ is also almost \v{C}ech-complete, hence Baire (see \cite{AartsLut}). Another argument: The original norm of $E$ has the Kadec-Klee property, i.e. the weak  and the norm topologies coincide on the unit sphere of $E$; now apply \cite[Proposition 12.56]{fabian}.   Since $\ell'_1 =\ell_\infty$ is not separable, $B_w$ is not metrizable, and hence $B_w$ is not a $k$-space by Theorem \ref{t:Ball-Metriz}. In particular, the conditions ``$B_w$ is a $k$-space'' in Theorem \ref{t:S-W-2} and ``$B_w$ is metrizable'' in Theorem \ref{t:E-W} cannot be removed. Note that $E$ has $RNP$ and $(PC)$ but it is not a Godefroy space (see \cite{edgar}).}
\end{example}

\begin{example}{\em
Let $E=c_0$. As $c'_0=\ell_1$ is separable,  $E$ is a weakly $\aleph_0$-space by \cite{GKKLP} and $B_w$ is metrizable (hence a $k$-space). The ball $B_w$ is not \v{C}ech-complete by \cite{edgar}. Note that (see \cite{edgar}) $E$ does not have $RNP$ and $(PC)$ and is not a Godefroy space, but $E$ is Asplund.}
\end{example}


Recall that a topological space $X$ is called \emph{cosmic}, if $X$ is a regular space with a countable network (a family $\mathcal{N}$ of subsets of $X$ is called a \emph{network} in $X$ if, whenever $x\in U$ with $U$ open in $X$, then $x\in N\subset U$ for some $N \in\mathcal{N}$). A space $X$ is called a {\em $\sigma$-space} if it is regular and has a $\sigma$-locally finite network.
\begin{remark}
Let $E$ be the James Tree space. Then $E_w$ is a cosmic space as the continuous image of a separable metrizable space $E$ (see \cite{Mich}). Since $E'$ is not separable, $E'$ has uncountable $cs^\ast$-character. So one cannot replace the countability of $cs^\ast$-character of $E_w$ or $B_w$ in Theorem \ref{t:Ball-Metriz} by $E_w$ or $B_w$ being  cosmic.
\end{remark}

For a Banach space $E$, Corson \cite{Corson} proved that $E_w$ is paracompact if and only if $E_w$ is Lindel\"{o}f, and Reznichenko \cite{Reznich} proved that $E_w$ is Lindel\"{o}f if and only if $E_w$ is normal. On the other hand, we proved in \cite{GK-GMS1} that a topological space $X$ is cosmic (resp. an $\aleph_0$-space) if and only if it is a Lindel\"{o}f $\sigma$-space (resp. a Lindel\"{o}f $\aleph_0$-space). In particular, these results imply
\begin{proposition}
For a Banach space $E$, the following conditions are equivalent: {\em (a)} $E_w$ is an $\aleph_0$-space;  {\em (b)} $E_w$ is a paracompact $\aleph$-space;  {\em (c)} $E_w$ is a Lindel\"{o}f  $\aleph$-space;  {\em (d)} $E_w$ is a normal  $\aleph$-space.
\end{proposition}

\bibliographystyle{amsplain}

\begin{thebibliography}{10}

\bibitem{AartsLut}
J. Aarts, D. Lutzer, \emph{Completeness properties designed for recognizing Baire spaces}, Dissertationes Math. \textbf{116} (1974), 48 pp.

\bibitem{Arhangel}
A. V. Arhangel'skii, \emph{Topological function spaces}, Math. Appl. \textbf{78}, Kluwer Academic Publishers, Dordrecht, 1992.


\bibitem{Arch-Pon}
A. V. Arhangel'skii, V. I. Ponomarev, \emph{On dyadic bicompacta}, Soviet Mathematics, Doklady, \textbf{9}  (1968) 1220--1224, translated by F. Cezus from the Russian original.



\bibitem{Banakh1}
T.~Banakh, A. Leiderman \emph{The strong Pytkeev property in topological spaces}, (http://arxiv.org/abs/1412.4268).

\bibitem{BZ}
T.~Banakh, L.~Zdomskyy, \emph{The topological structure of (homogeneous) spaces and groups with countable $cs^\ast$-character}, Applied General Topology \textbf{5} (2004), 25--48.

\bibitem{bierstedt2}
K. D. Bierstedt, J. Bonet, \emph{Some aspects of the modern theory of Fr\'{e}chet spaces}, RACSAM Rev. R. Acad. Cien. Serie A. Mat. \textbf{97} (2003), 159--188.



\bibitem{Corson}
H. Corson, \emph{ The weak topology of a Banach space}, Trans. Amer. Math. Soc. \textbf{101} (1961), 1--15.

\bibitem{edgar}
G. A. Edgar, R. F. Wheeler, \emph{Topological properties of Banach spaces}, Pacific J. Math. \textbf{115} (1984), 317--350.

\bibitem{Eng}
R.~Engelking, \emph{General topology}, Panstwowe Wydawnictwo Naukowe, Waszawa, 1977.

\bibitem{fabian} M. Fabian, P. Habala, P. H\'{a}jek, V. Montesinos, J. Pelant, V. Zizler, \emph{Functional Analysis and Infinite-Dimensional Geometry}, Canadian Mathematical Society, Books in Mathematics (2001).


\bibitem{GaK}
S.~Gabriyelyan, J.~K{\c{a}}kol, \emph{Metrization conditions for topological vector spaces with Baire type properties}, Topology Appl.  \textbf{173} (2014), 135--141.

\bibitem{GK-GMS1}
S.~Gabriyelyan, J.~K\c akol, \emph{ On $\mathfrak{P}$-spaces and related concepts},
preprint (http://arxiv.org/abs/1412.1494).

\bibitem{GKKM}
S. Gabriyelyan, J. K{\c{a}}kol, W. Kubi\'s, W. Marciszewski, \emph{ Networks for the weak topology of Banach and Fr\'echet spaces}, preprint (http://arxiv.org/abs/1412.1748).


\bibitem{GKKLP}
S.~Gabriyelyan, J.~K{\c{a}}kol, A. Kubzdela, M. Lopez Pellicer, \emph{On topological properties of Fr\'{e}chet locally convex spaces with the weak topology},  Topology  Appl., accepted.

\bibitem{GKL2}
S.~Gabriyelyan, J.~K{\c{a}}kol, A.~Leiderman,  \emph{The strong Pytkeev property for topological groups and topological vector spaces}, Monatsch. Math. \textbf{175} (2014), 519--542.


\bibitem{gruenhage}
G. Gruenhage, Generalized metric spaces, \emph{Handbook of Set-theoretic Topology}, North-Holland, New York, 1984, 423--501.

\bibitem{horvath}
J. Horv\'{a}th, \emph{Topological Vector Spaces and Distributions, I}. Addison-Wesley, Reading, Mass, 1966.

\bibitem{kak}
J.~K\c{a}kol, W.~Kubi\'s, M.~Lopez-Pellicer, \emph{Descriptive Topology in Selected  Topics of Functional Analysis}, Developments in Mathematics, Springer, 2011.




\bibitem{KocScheep}
Lj. D. Ko\v{c}inac, M. Scheepers, \emph{Combinatorics of open covers (VII): Groupability}, Fund Math. \textbf{179} (2003), 131--155.


\bibitem{malyhin}
V.~I.~Malykhin, G.~Tironi,  \emph{Weakly Fr\'echet-Urysohn and Pytkeev spaces}, Topology Appl. \textbf{104} (2000), 181--190.


\bibitem{OMe2}
P. O'Meara,  \emph{On paracompactness in function spaces with the compact-open topology}, Proc. Amer. Math. Soc. \textbf{29} (1971), 183--189.

\bibitem{Mich}
E.~Michael, \emph{$\aleph_0$-spaces}, J. Math. Mech. {\bf 15} (1966), 983--1002.

\bibitem{bonet}
P. P\'{e}rez Carreras, J. Bonet, \emph{Barrelled Locally Convex Spaces}, North-Holland Mathematics Studies \textbf{131}, North-Holland, Amsterdam, 1987.


\bibitem{nyikos}
P. J. Nyikos, \emph{Metrizability and the Fr\'{e}chet--Urysohn property in topological groups},  Proc. Amer. Math. Soc.  \textbf{83}  (1981),  793--801.

\bibitem{Pyt}
E.~G.~Pytkeev,  \emph{On maximally resolvable spaces}, Proceedings of the Steklov Institute of Mathematics \textbf{154} (1984), 225--230.


\bibitem{Reznich}
E. A. Reznichenko, \emph{Normality and collective normality of function spaces}, Moscow Univ. Math. Bull. \textbf{45} (1990), no. 6, 25--26.

\bibitem{S-W}
G. Schl\"{u}chtermann, R. F. Wheeler, \emph{The Mackey dual of a Banach space}, Noti de Matematica, \textbf{XI} (1991), 273--287.


\bibitem{boaz}
B.~Tsaban, L.~Zdomskyy,  \emph{On the Pytkeev property in spaces of continuous functions (II)}, Houston J. of Math. \textbf{35} (2009), 563--571.

\end{thebibliography}

\end{document}